\documentclass[12pt]{amsart}
\usepackage{amsfonts}
\usepackage{amsmath}
\usepackage{amssymb,latexsym}
\usepackage[mathcal]{eucal}
\usepackage{amscd}
\usepackage[pdftex,bookmarks,colorlinks,breaklinks]{hyperref}

\input xy
\xyoption{all}

\newcommand{\Acal}{\mathcal{A}}
\newcommand{\Bcal}{\mathcal{B}}
\newcommand{\Ccal}{\mathcal{C}}

\newcommand{\Fcal}{\mathcal{F}}

\newcommand{\Hcal}{\mathcal{H}}

\newcommand{\ch}{\mathbf{1}}

\newcommand{\Z}{\mathbb{Z}}
\newcommand{\Q}{\mathbb{Q}}
\newcommand{\R}{\mathbb{R}}
\newcommand{\C}{\mathbb{C}}
\newcommand{\N}{\mathbb{N}}
\newcommand{\T}{\mathbb{T}}
\newcommand{\E}{\mathbb{E}}
\newcommand{\A}{\mathbb{A}}

\newcommand{\Xb}{\mathbf{X}}
\newcommand{\Yb}{\mathbf{Y}}
\newcommand{\Zb}{\mathbf{Z}}

\newcommand{\al}{\alpha}
\newcommand{\Ga}{\Gamma}
\newcommand{\ga}{\gamma}
\newcommand{\del}{\delta}
\newcommand{\Del}{\Delta}
\newcommand{\ep}{\epsilon}
\newcommand{\sig}{\sigma}

\newcommand{\la}{\lambda}

\newcommand{\om}{\omega}
\newcommand{\Om}{\Omega}

\newcommand{\br}{\vspace{3 mm}}
\newcommand{\imp}{\Rightarrow}

\newcommand{\tri}{\bigtriangleup}

\newcommand{\cls}{{\rm{cls\,}}}

\newcommand{\Iso}{{\rm{Iso\,}}}

\newcommand{\supp}{{\rm{supp\,}}}

\newcommand{\spann}{{\rm{span}}}

\newcommand{\Homeo}{\rm{Homeo\,}}

\newcommand{\vertiii}[1]{{\left\vert\kern-0.25ex\left\vert\kern-0.25ex\left\vert #1 
    \right\vert\kern-0.25ex\right\vert\kern-0.25ex\right\vert}}

\swapnumbers
\theoremstyle{plain}
\newtheorem{thm}{Theorem}[section]

\newtheorem{prop}[thm]{Proposition}

\theoremstyle{definition}
\newtheorem{defn}[thm]{Definition}

\newtheorem{rmk}[thm]{Remark}
\newtheorem{rmks}[thm]{Remarks}

\newtheorem{exa}[thm]{Example}

\newtheorem{prob}[thm]{Problem}

\begin{document}

\title[]
{Weak mixing properties for nonsingular actions 
}

\author{Eli Glasner and Benjamin Weiss}

\address{Department of Mathematics\\
     Tel Aviv University\\
         Tel Aviv\\
         Israel}
\email{glasner@math.tau.ac.il}

\address {Institute of Mathematics\\
 Hebrew University of Jerusalem\\
Jerusalem\\
 Israel}
\email{weiss@math.huji.ac.il}

\begin{abstract}
For a general group $G$ we consider various weak mixing properties of nonsingular actions. 
In the case where the action is actually measure preserving all these properties coincide, and our purpose here is to check which implications persist in the
nonsingular case.
\end{abstract}




\date{December 14, 2014}


\maketitle

\tableofcontents

\section*{Introduction}

Let $G$ be a second countable locally compact topological group. We are interested here in
various weak mixing properties of nonsingular actions of $G$. A {\em nonsingular}
(or {\em quasi-invariant}) action is a measurable action of $G$ on a standard Lebesgue probability space $(X,\Bcal,\mu)$, 
where the action preserves the measure class of $\mu$ (i.e. $\mu(A) = 0 \iff \mu(gA) =0$ 
for every $A \in \Bcal$ and every $g \in G$).
In the case where the action is actually measure preserving (i.e. $g\mu =\mu$ for every $g \in G$)
all of these weak mixing properties coincide, and our purpose is to check which implications 
persist in the nonsingular case.

More specifically the mixing conditions we consider are defined as follows:


\begin{defn}
Let $\Xb = (X,\Bcal,\mu,\{T_g\}_{g \in G})$ be a nonsingular $G$ action.
\begin{enumerate}
\item
$\Xb$ is {\em ergodic} if every $G$-invariant $B \in \Bcal$ (i.e.
$\mu(gB \tri B)=0$ for every $g \in G$) satisfies $\mu(B) =0$ or $1$. 
\item
$\Xb$ is {\em doubly ergodic}, (DE) if the product system $\Xb \times \Xb$, where $G$ acts diagonally on 
$X \times X$, is ergodic.
\item
$\Xb$ is {\em ergodic with isometric coefficients}, (EIC) if there is no non-constant
Borel measurable equivariant map $\phi : X \to Z$, where $(Z,d)$ is a separable metric space 
on which $G$ acts by isometries.
\item
$\Xb$ is {\em ergodic with unitary coefficients}, (EUC) if there is no non-constant
Borel measurable equivariant map $\phi : X \to H$, where $H$ is a separable Hilbert space 
on which $G$ acts by unitary operators.
\item
$\Xb$ is {\em weakly mixing}, (WM) if it has the multiplier property for probability measure preserving 
ergodic $G$-actions; i.e. for every such p.m.p. action $\Yb = (Y,\Ccal,\nu,\{S_g\}_{g \in G})$ the product system
$\Xb \times \Yb$ is ergodic.
\item
$\Xb$ is {\em $L_\infty$ weakly mixing} ($L_\infty$-WM) if there are
no nontrivial invariant finite dimensional subspaces of $L_\infty(X,\mu)$.
\end{enumerate}
\end{defn}

\begin{rmks}
\begin{enumerate}
\item
The notions of EIC and EUC (and other close variations) appeared first in \cite{BM}, and then also in \cite{Mo}, \cite{Ka3}, \cite{MS} and \cite{BF}.
\item
In a series of related works (e.g. \cite{BFMS}, \cite{GIILS} and \cite{JKLSS})
the authors treat
a notion of ``double ergodicity" which, unfortunately, is different from the one we adopt here
(see Remark \ref{silva} below).
\item
Two more related works \cite{Glu} and \cite{Bj} appeared recently in the Arxiv.
\item
$L_\infty$-eigenfunctions of nonsingular transformations were first introduced by M. Keane
who showed that for a nonsingular $\Z$-action $(X,\Bcal,\mu,T)$ and a probability 
measure preserving $\Z$-action $(Y, \Ccal,\nu,S)$ the product system
$(X \times Y, \mu \times \nu, \Bcal \times \Ccal, T \times S)$ is ergodic iff $\sig_0(e(T)) =0$.
Here $\sig_0$ is the (restricted) maximal spectral type of $S$ and $e(T)$
is the collection of $L_\infty$ eigenvalues of $T$ (see \cite[Theorem 2.7.1]{Aa}). 
From this theorem one can easily deduce the equivalence of WM with
$L_\infty$-WM for $\Z$-actions. 
\item
%
The requirement in the definition of the EIC property that the metric space $Z$ be separable is in
fact redundant as it can be shown that the subset $\cls \supp \phi_*(\mu)$ is necessarily separable.
\end{enumerate}
\end{rmks}

In the first section we sort these conditions according to their strength, and show that they
are all equivalent when $\mu$ is measure preserving. In the second
section we give a short proof of a theorem of Gl\"{u}cksam's \cite{Glu} which says
that when the acting group is locally compact group all of whose irreducible unitary representations
are finite dimensional (a Moore group) WM is equivalent to $L_\infty$-WM.

Strongly almost transitive (SAT) systems are nonsingular systems with the property that every positive set can be blown up to almost full measure. 
In section 3 we show that SAT systems are EIC and give an example of a SAT (hence EIC) 
system which is not DE.

For a well behaved probability measure $m$ on a locally compact group $G$, every $m$-stationary system $\Xb$ (i.e. one in which $m*\mu = \mu$) is nonsingular. 
We treat this class of stationary  systems
in section 4 and, using a strengthened form of a theorem of Kaimanowich, show that
such systems are DE (in fact doubly EIC). 

In sections 5 and 6 we use our previous results
to show that the free group $F_2$ and even the group of integers $\Z$ admit 
nonsingular actions which are EIC but not DE. For $F_2$ this is its Poisson boundary
and for $\Z$ the example arises from a standard planar random walk.
In the final section we add one more characterization to Furstenberg's list of 
equivalent conditions on a stationary system to be $F_m$-proximal. 

\section{A chain of implications}

\begin{thm}\label{imp}
For nonsingular $G$ systems $\Xb = (X,\Bcal,\mu,\{T_g\}_{g \in G})$ we have:
\begin{gather*}
DE \ \imp EIC  \ \imp EUC = WM  \ \imp L_\infty{\text{-}}WM \ \imp  {\text {ergodicity}}.
\end{gather*}
If the $G$ action on $\Xb$ is measure preserving then all these properties (except for
the last one, ergodicity) are equivalent. 
\end{thm}


\begin{proof}
$DE \ \imp EIC$. 
Suppose $\Xb$ is doubly ergodic and let $\phi : X \to Z$ be a metric factor. Then the function
$$
(x,x') \to d(\phi(x),\phi(x'))
$$
is an invariant measurable function, hence a constant $t_0 \ge 0$. 
Fix a point $z_0$ in the support of the measure $\phi_*(\mu)$. For any
$r>0$ the ball $B(z_0,r)$ has positive $\phi_*(\mu)$ measure
and it then follows from the constancy $\mu \times \mu$-a.e. of the
distance function that $t_0$ is less than $2r$. Since this holds
for every $r > 0$ it follows that $t_0 = 0$ and
this clearly implies that $\phi$ is constant.

The implications $EIC \imp EUC$ is clear. 
We will show the equivalence of EUC and WM in a separate statement, Theorem \ref{main} below.

To see that the implication 
$$
EUC \imp L_\infty{\text{-}}WM
$$ 
holds, observe that if $V \subset L_\infty(X,\mu)$ is a 
nontrivial (i.e. of dimension $\ge 1$) finite dimensional invariant subspace, then $G$ acts on the compact unit sphere $S_V$ of $V$ by
isometries and thus as a precompact group. (The group of isometries of a compact metric space
equipped with the topology of uniform convergence is a compact group.)
Let $\tilde G$ denote the closure of the image of $G$ in the group
$\Iso(V)$ and let $\la$ denote the normalized Haar measure on the 
compact topological group $\tilde G$.
Pick some Hermitian sesquilinear form $\rho$ on $V$ 
(of course we can assume that $V$ is defined over $\C$) and set
$$
\tilde\rho (u, v) = \int_{\tilde{G}} \rho(gu, gv) \, d\la(g)
$$
Then $\tilde \rho$ is a $G$-invariant inner product on $V$ and $G$ acts on the Hilbert space $(V, \tilde\rho)$ as a group of unitary operators.
Next choose a basis for  $(V, \tilde\rho)$, say $\{f_1, f_2, \dots, f_k\}$. 
We can assume that each $f_i$ is defined everywhere and then, for each $x \in X$,
the assignment $f_i \mapsto f_i(x), i=1,2,\dots,k$, extends uniquely
to a linear functional $\phi(x) \in  H := (V, \tilde\rho)^*$.
Via the form $\tilde\rho$ the latter vector space becomes a Hilbert space
and the unitary $G$-action on $(V, \tilde\rho)$ induces a unitary 
$G$-action on $H$.
One can easily check that the map $\phi : X \to H$
is a non-constant measurable $G$-equivariant map, contradicting the EUC property of $\Xb$.

Suppose now that $\mu$ is ergodic and that in the measure class of $\mu$ there is an invariant probability measure (a unique one, by ergodicity).
So we can assume that $\mu$ itself is invariant.
 

If $V$ is a finite dimensional $G$-invariant subspace of $L_2(\mu)$,
choose an orthonormal basis $\{f_1, f_2,\dots,f_n\}$ for $V$
and observe that the function $\phi : X \to \C$ defined by 
$\phi(x) = \sum_{i=1}^n |f(x)|^2$ is non-zero and $G$-invariant (the action is
via unitary matices). By ergodicity $\phi(x) = C > 0$ is a constant a.e.
and it follows that the functions $f_i$ are bounded,
%
so that $V \subset L_\infty(\mu)$.
Thus it follows that if $\Xb$ is $L_\infty$-WM then $L_2(\mu)$
admits no nontrivial finite dimensional invariant subspaces
(one can call this latter property $L_2$-WM).
As is well known $L_2$-WM implies weak mixing and, in turn,
this implies that $\Xb$ is DE (as $\Xb$ is measure preserving.) We conclude that 
for measure preserving systems all the 
properties DE,  EIC,  EUC, WM, $L_2$-WM, and $L_\infty$-WM
are equivalent.
(See e.g. \cite[Theorem 3.11]{G}).

Finally, the implication $L_\infty$-WM\  $\ \imp \ $ ergodicity is clear. 
\end{proof}

\begin{thm}\label{main}
A nonsingular system $\Xb = (X,\Bcal,\mu,\{T_g\}_{g \in G})$ is weakly mixing iff it is EUC.
\end{thm}

\begin{proof}
1. Assume first that $\Xb$ is EUC and let $\Yb$ be a probability measure preserving $G$-action.
Suppose $W \subset X \times Y$ is a $G$-invariant Borel set and
set $F(x,y) = \ch_W$. Then the map $\phi : X \to L_2(Y,\nu)$ defined by
$\phi(x) = F_x := F(x,\cdot)$ is
clearly measurable and, denoting by $U_g$ the Koopman operator induced by $S_g$
on $L_2(Y,\nu)$ (explicitly, $U_gf(y) = f(S_{g^{-1}}y)$) we have:
\begin{align*}
\phi(T_g x)(y) & = F(T_g x, y) = F(T_{g^{-1}}T_g x, S_{g^{-1}} y) \\
& =
F_x(S_{g^{-1}} y) = (U_g F_x)(y) \\
& = (U_g \phi(x))(y). 
\end{align*}
Hence, for $\mu$-a.e. $x$
$$
\phi(T_g x) = U_g \phi(x).
$$
Since we are assuming that $\Xb$ is EUC, we conclude that $\phi(x) \equiv f$ is a constant $\mu$-a.e.
Then, 
$$
f(S_gy) = F(x,S_gy) = F(T_{g^{-1}}x,y) = f(y)
$$ 
for $\nu$-a.e. $y$ and by the ergodicity of $\Yb$
we conclude that either $f \equiv 0$ or $f \equiv 1$; i.e. $\mu\times \nu(W)$ is either $0$ or $1$.
Thus, the $T \times S$ action is indeed ergodic.

\br

2. We now assume that $\Xb$ is weakly mixing and suppose that $\phi : X \to H$ is a 
measurable equivariant map, where $H$ is a separable Hilbert space 
on which $G$ acts by unitary operators $\{U_g\}_{g \in G}$.
Thus $\mu$-a.e.
$$
\phi(T_gx) = U_g \phi(x).
$$
We assume with no loss in generality that $H = \cls\ \spann( \supp \phi_*(\mu))$. 
Let $H_0 = \{v \in H : U_g v = v,\ \forall g\in G\}$ and $H_1 = H_0^{\perp} =
H \ominus H_0$. If $H_1=0$ then clearly $\phi$ is a constant 
(clearly WM implies ergodicity) and we are done.
So we can now assume that $H = H_1$, i.e. the representation 
$g \mapsto U_g$ on $H$ admits no nonzero fixed vectors.  

By passing to a cyclic subspace we can assume that there is a vector $v_0 \in H$
such that $H = {\cls}{\spann} \{U_g v_0 : g \in G\}$. 
Set
$$
c(g) = \langle U_g v_0,v_0 \rangle,
$$
and let $(\Om,\Fcal,P, \{X_g\}_{g \in G})$ be an associated Gauss process; i.e.
the collection $\{X_g\}_{g \in G}$ is a set of centered random Gauss variables defined on the probability space
$(\Om,\Fcal,P)$ with correlation function
$$
c(gh^{-1}) = \E(X_gX_h).
$$
Then, the translations defined on the $X_g$'s by $h(X_g) = X_{gh}$ define an action of 
$G$ on $(\Om,\Fcal,P)$. Denoting this action by $\{R_g\}_{g \in G}$, the system
$(\Om,\Fcal,P, \{R_g\}_{g \in G})$ becomes a probability measure preserving action of $G$.
The fact that $H$ contains no nonzero fixed vector implies that the system
$(\Om,\Fcal,P, \{R_g\}_{g \in G})$ is ergodic (see e,g. \cite[Theorems 3.4 and 3.59]{G}).
By our assumption then the product action $(X \times \Om, \Bcal \otimes \Fcal, \mu \times P,
\{T_g \times R_g\}_{g \in G})$ is ergodic. 

Now, by the construction of the Gauss process $\{X_g\}_{g\in G}$,
there is a unitary equivalence between $H$ and the {\em first Wiener chaos} 
$\Hcal : = {\cls}{\spann} \{X_g : g \in G\} \subset L_2(\Om,P)$, say,
$$
V : H \to \Hcal 
$$
with $U_{R_g}V = V U_g$ for every $g \in G$. 
Set $w(x,\om) = V \phi(x)(\om)$, then
\begin{align*}
w(T_{g^{-1}} x,\om) &  = V \phi(T_{g^{-1}} x)(\om) = V U_g \phi(x)(\om) \\ 
& = U_{R_g} V \phi(x)(\om)=   U_{R_g} w(x,\om) \\
& = w(x, R_g\om), 
\end{align*}
whence
$$
w(T_{g^{-1}} x,R_{g^{-1}} \om) = w(x,\om) 
$$
for every $g \in G$. By our assumption then $w(x,\om)=V \phi(x)(\om)$ is a constant 
$\mu \times P$ a.e. and this leads to a contradiction if $H$ is nontrivial.
\end{proof}

\begin{rmk}
Our proof of Theorem \ref{main} was motivated by \cite{Be}, and by the proof of Proposition 4.6 
in \cite{ALW}, which is the main ingredient of the proof of Theorem 4.7 in \cite{ALW}.
Of course Theorem \ref{main} is a far reaching generalization of Theorem 4.7 in \cite{ALW}. 
For the implication EUC $\imp$ WM see also \cite[Proposition 2.4]{MS}
\end{rmk}

\begin{prob}
Find an example of a system which is EUC but not EIC.
\end{prob}

\br

\section{Moore groups}
A locally compact group $G$ which has the property that all of its irreducible unitary representations
are finite dimensional is called a {\em Moore group} (see
\cite{M}, \cite{T}). Of course abelian groups are Moore groups and in \cite{M}
it is shown that a locally compact group $G$ is a Moore group if and only if
$G = {\rm proj} \lim G_\al$ where each $G_\al$ is a Lie group which contains
an open subgroup $H_\al$ of finite index which is a $Z$-group. Recall that
a topological group $H$ is a {\em $Z$-group} if $H/Z(H)$ is compact, where $Z(H)$ is the center of $H$
(\cite[Theorem 3]{M}). 

In a recent work A. Gl\"{u}cksam \cite[Theorem 4.3]{Glu} proves that for a second countable Moore group 
WM is equivalent to $L_\infty{\text{-}}WM$. We will next use Theorem \ref{main} to 
obtain a much shorter proof of this result.

%


\begin{thm}
For a second countable Moore group the property 
EUC (hence also WM) is equivalent to the property $L_\infty{\text{-}}WM$.
\end{thm}

\begin{proof}
By Theorem \ref{main}, EUC = WM.
Thus the implication EUC $\imp L_\infty{\text{-}}WM$ was already proved above in 
Theorem \ref{imp}.
For the other direction assume that $\Xb$ has the property $L_\infty{\text{-}}WM$ and 
assume to the contrary that $\phi : X \to H$ is a nontrivial measurable equivariant map into a separable Hilbert space
on which the group $G$ acts via unitary operators $\{U_g : g \in G\}$.
Again, with no loss in generality, assume that $H = \cls\ \spann( \supp \phi_*(\mu))$. 
We observe that by the ergodicity of $\Xb$ and the assumption that
$G$ acts on $H$ by unitary operators, it follows that the map $x \mapsto \|\phi(x)\|$
is a constant $\mu$-a.e.; since this constant is nonzero we can assume that it is $1$.

Let us first consider the case where $H$ is irreducible and finite dimensional. We then choose 
$h_0 \in H$ with $\|h_0\|=1$ such that the function $f_0 : X \to D$, where $D$
is the unit disk in $\C$, defined by $f_0(x) = \langle \phi(x), h_0 \rangle$ is nonconstant.
Now, $f_0$ is a nonconstant function in $L_\infty(X,\mu)$ and
the space $V = \spann \{g \cdot f_0 : g \in G\}$ is a $G$-invariant subspace of $L_\infty(X,\mu)$.
We have
$$
(g \cdot f_0)(x) = f_0(g^{-1}x) =  \langle \phi(g^{-1}x), h_0 \rangle 
= \langle U_g(\phi(x)), h_0 \rangle =  \langle \phi(x), U_{g^{-1}}h_0 \rangle.
$$
Since for an element $h \in H$ the map $f(x) = \langle \phi(x),h \rangle$ satisfies
the inequality $\|f\|_\infty \le \|h\|$ it follows that the natural map
from $H = \spann \{U_g h_0: g \in G\}$ onto $V$, which sends $h_0$ to $f_0$, is
an isomorphism. Thus $V$ is indeed a nontrivial
finite dimensional $G$-invariant subspace of $L_\infty(X,\mu)$.

The next step of the proof 
is based on the fact that for 
a second countable Moore group $G$ any unitary representation $\pi$ on a separable Hilbert space 
$H$ can be represented as a direct integral  $\pi \cong \int^{\oplus} \pi_t\, dP(t)$, on some parameter space $T$ equipped with a measure $P$, where each $\pi_t$ is an irreducible representation on a finite dimensional Hilbert space $H_t$, so that $H \cong \int^\oplus H_t \, dP(t)$. 
See \cite{Mac} and \cite{D}.
It now follows that if $\phi : X \to H$ is a nontrivial map as above then for a set $A \subset T$
of positive $P$-measure, for every $t \in A$, the composition $\phi_t$ of $\phi$ with the projection onto the component $H_t$, is again a nontrivial map.  
\end{proof}

\br


\section{SAT dynamical systems}

Strongly approximately transitive group actions were first introduced and studied by 
W. Jaworski, \cite{J}.

\begin{defn}
A nonsingular action $\Xb = (X,\Bcal,\mu,\{T_g\}_{g \in G})$ is
{\em strongly almost transitive (SAT)} if for every measurable $A \subset X$
with $\mu(A) >0$, there is a sequence $\{g_n\}\subset G$ with $\mu(g_nA) \to 1$.
\end{defn}

\begin{prop}
\begin{itemize}
\item 
Every SAT system is ergodic.
\item
A SAT system admits no non-identity endomorphism (automorphism).
\item
In every SAT system the center of $G$ acts trivially.
\item
Nilpotent (in particular abelian) groups admit no nontrivial SAT systems.
\end{itemize}
\end{prop}

\begin{proof}
The first assertion is clear. Suppose $\phi : X \to X$ is a non-identity endomorphism
of the SAT system $\Xb$. Thus, $\phi$ is measurable, $G$ equivariant, 
with $\phi_*(\mu) \sim \mu$.
Let $B \in \Bcal$ satisfy $0 < \mu(B)$ and $\mu(B \cap \phi(B)) =0$.
Choose a sequence $g_n \in G$ such that $\mu(g_nB) \to 1$.
Then also $\mu(g_nSB) = \mu(Sg_nB) \to 1$, hence eventually,
$\mu(g_n B \cap S g_n B) = \mu(g_n(B \cap SB)) > 0$; a contradiction.
The last two assertions follow readily.
\end{proof}


Note that if $\Xb$ is nontrivial SAT then the product action on $X \times X$ is never
SAT. In fact for $A \subset X$ with $0 <\mu(A) < 1$ we have
$\mu \times \mu (T_g A \times T_g A^c) = \mu(T_gA)(1 - \mu(T_gA)) \le \frac14$
for every $g \in G$.

%
%

\begin{prop}\label{SAT-EIC}
If $\Xb=(X,\mathcal{B},\mu,G)$ is SAT then it is EIC.
\end{prop}

\begin{proof}
Let $\Xb$ be a SAT system and assume to the contrary that 
$\phi : X \to Z$ is a non-constant,
Borel measurable, equivariant map, where $(Z,d)$ is a separable metric space 
on which $G$ acts by isometries. We assume (as we well may) that $Z = \cls \phi(X)$,
so that $\supp(\nu) = Z$, where $\nu =\phi_*(\mu)$.
We choose an $\ep >0$ and open balls $U = B(z_0, \ep)$ and $V = B(z_1, \ep)$ in $Z$
whose centers satisfy $d(z_0, z_1) > 2\ep$.
Let $A = \phi^{-1}(U)$ and $B = \phi^{-1}(V)$.
Then the sets $A, B \in \Bcal$ have $\mu$-measures strictly between zero and one and 
by the SAT property there is $g \in G$ such that $\mu(gA) > \max (1 - \mu(A), 1 - \mu(B))$.
It follows that both $\mu(gA \cap A) > 0$ and $\mu(gA \cap B) > 0$. Therefore
also $\nu(gU \cap U) >0$ and $\nu(gU \cap V) > 0$, but this contradicts the fact
that $g$ is an isometry on $Z$. 
\end{proof}

\begin{rmk}\label{silva}
The condition on $\Xb$ that for any two positive sets $A, B \in \Bcal$
there is an element $g \in G$ with both $\mu(gA \cap A) > 0$ and $\mu(gA \cap B) >0$,
is called by Silva et.al. ``double ergodiity"
(see \cite{BFMS}, \cite{GIILS} and \cite{JKLSS}).
\end{rmk}

On the other hand we will next see that the SAT property does not necessarily imply double ergodicity. (See also Propositions \ref{F2exa} and \ref{Zexa} below.)

\begin{exa}\label{SATnDE}
Let $G = \R \rtimes \R^+$ be the ``$ax +b$" group. As shown in \cite{J} the natural action of this group
on $\R$ has the property that any absolutely continuous probability measure $\mu$ on $\R$ is SAT. 
(It is also easy to see that the SAT property follows directly from the existence of density points in any Borel
measurable subset  $A \subset \R$ of positive Lebesgue measure.)
Of course such a measure will be SAT also with respect to any countable dense subgroup of $G$.
For concreteness let us consider the group $\Ga = \Q \rtimes \Q^+$, and let $\mu$ be
any probability measure which is in the class of Lebesgue measure on $\R$.
Taking $A,B, C$ and $D$ to be any four disjoint consecutive open intervals in $\R$, we
see that there is no $\ga \in \Ga$ for which $\mu(\ga B \cap D) >0$ and also
$\mu(\ga C \cap A) >0$. It is not hard to check that the cartesian square
$(\R \times \R, \mu \times \mu, \Ga)$ has exactly two ergodic component
$\{(x, y ) : x < y \}$ and $\{(x,y) : x > y\}$. Compare this with 
Example \ref{exa,F2}, where the cartesian square decomposes into a continuum of ergodic 
components according to the value of the cross-ratio;
and with the Example in Section \ref{Sec,Z}, where the cartesian square is dissipative.
\end{exa}

We thus have:
 
\begin{prop}\label{affine}
The nonsingular system $(\R,\mu,\Ga)$ is SAT (hence EIC) but not doubly ergodic.
\end{prop}

 \br

\section{$m$-stationary dynamical systems}\label{sec,m}
Let $G$ be a locally compact second countable topological group. A probability 
measure $m$ on $G$  is called {\em spread out} (or {\em etal\'{e}e})
if there exists a convolution power $m^{*n}$ which is not singular with respect to the Haar measure class on $G$.
The measure $m$ is {\em non-degenerate} if the minimal closed semigroup $S \subset G$ with 
$m(S) = 1$ is $G$, and it is {\em symmetric} if it is invariant under
the map $x \mapsto x^{-1},\ x \in G$.
The measure $m$ is {\em admissible} if it is both spread out and non-degenerate.

A probability measure $\mu$ on a Borel dynamical system $(X,\Bcal,G)$ is called  
{\em $m$-stationary} when the equation $m*\mu=\mu$ holds. 
When $m$ is admissible it follows that the system $\Xb=(X,\Bcal,\mu,G)$ is nonsingular; 
i.e. the measure $\mu$ is quasi-invariant: $\mu(gA) =0 \iff \mu(A)=0$ 
for every $A \in \Bcal$ and $g \in G$ (\cite[Lemma 1.1]{NZ}).
 
We can always find a {\em topological model} for an
$m$-system $\Xb$, meaning that $X$ can be chosen to be a compact
metric space on which $G$ acts by homeomorphisms
(see \cite[Theorem 3.2]{V}).

With the measure $m$ one associates a random walk on $G$ as follows.
Let $\Omega=G^\N$ and let $P=m^\N=m\times m \times m \dots$ be the
product measure on $\Omega$, so that $(\Om,P)$ is a probability space.
We let $\xi_n:\Om\to G$, denote the
projection onto the $n$-th coordinate,\ $\ n=1,2,\dots$. We refer
to the stochastic process $(\Om,P,\{\eta_n\}_{n\in\N})$, where
$\eta_n=\xi_1\xi_2\cdots\xi_n$ as the {\em $m$-random walk on $G$}.

A real valued function $f(g)$ for which $\int f(gg')\,dm(g')=f(g)$
for every $g \in G$ is called {\em harmonic}. For a harmonic $f$ we have
\begin{gather*}
E(f(g\xi_1\xi_2\cdots\xi_n\xi_{n+1}|\xi_1\xi_2\cdots\xi_n)\\
=\int f(g\xi_1\xi_2\cdots\xi_n g')\,dm(g') \\ =
f(g\xi_1\xi_2\cdots\xi_n),
\end{gather*}
so that the sequence $f(g\xi_1\xi_2\cdots\xi_n)$ forms a martingale.

If ${\bf{X}} =(X,\mathcal{B},\mu,G)$ is an $m$-stationary system on a compact metric space $X$,
for $F \in C(X)$ let $f(g)= \int F(gx)\,d\mu(x)$. Then the equation 
$m*\mu=\mu$ shows that $f$ is harmonic. 
It is shown (e.g.) in \cite{F1} how these facts combined with the martingale convergence theorem lead to the following:

\begin{thm}
The limits
\begin{equation}\label{1}
\lim_{n\to\infty}\eta_n\mu=\lim_{n\to\infty}\xi_1\xi_2\cdots\xi_n\mu=\mu_\omega,
\end{equation}
exist for $P$ almost all $\om\in\Om$, and
$$
\int \mu_\om \, dP(\om) = \mu.
$$

\end{thm}

The measures $\mu_\om$ are the {\em conditional measures} of the 
$m$-system ${\bf{X}} $. We call the $m$-system ${\bf{X}}$,
{\em $m$-proximal} (or a {\em boundary} in the terminology of \cite{F1})
if $P$-a.s. the conditional measures $\mu_\om\in M(X)$ are point masses. 
It can be shown that this property does not depend on the topological model
chosen for $\Xb$.
Clearly a factor of a proximal system is proximal as well.
There exists a unique $m$-stationary system $\Pi(G,m)$, called the {\em Poisson boundary}
of $(G,m)$, which is a maximal boundary.  Thus an $m$-stationary system is $m$-proximal
if and only if it is a factor of $\Pi(G,m)$.
For more details and basic results concerning general $m$-stationary dynamical systems and, 
in particular, $m$-proximal systems we refer to \cite{FG1} and \cite{FG2}. 
We remind the reader that every $m$-proximal stationary system is SAT, \cite[Corollary 2.4]{J}
(see also \cite[Proposition 3.7]{FG2}).
For an alternative approach to the Poisson boundary see the seminal work of 
Kaimanovitch and Vershik \cite{KV}.

The following theorem is proved by Bj\"{o}rklund in \cite[Theorem 3.1]{Bj}.
As indicated by him a slightly different proof, based on the fact that WAP systems are stiff
(see \cite{FG1}), is available and is given below.

\begin{defn}
A measure preserving $G$-system $\Xb = (X,\mathcal{B},\mu,G)$ is said to have
discrete spectrum if $L_2(\mu)$ decomposes as a direct sum of finite dimensional
invariant subspaces. 
\end{defn}

Recall that an ergodic measure preserving system has discrete spectrum 
if and only if
the group $\{U_g : g \in G\}$, where $U_g$ is the unitary Koopman operator
on $L_2(\mu)$ defined by $U_gf(x) = f(g^{-1}x)$, is precompact in the 
strong operator topology (see e.g. \cite[Section 3.1]{G}).

\begin{defn}
A topological dynamical system $(X,G)$ (i.e. $X$ is a compact space
and $G$ acts on $X$ via a continuous homomorphism, say $\phi$, of $G$ into the group
$\Homeo(X)$ of self homeomorphisms of $X$ equipped with the uniform
convergence topology) is called {\em weakly almost periodic (WAP)}
if the closure of $\{\phi(g) : g \in G\} \subset X^X$, in the pointwise
convergence topology on $X^X$ (this closure is called the {\em enveloping
semigroup} of $(X,G)$), consists of continuous maps.
\end{defn}

\begin{thm}
Let $m$ be a non-degenerate spread out and symmetric probability measure on $G$.
An $m$-stationary system $\Xb=(X,\Bcal,\mu,G)$ is WM (as a nonsingular $G$-system)
iff it does not admit a nontrivial measure preserving factor which has discrete spectrum.
\end{thm}

\begin{proof}
Suppose first that $\Xb$ admits a nontrivial measure preserving factor 
$\Yb = (Y, \Ccal, \nu, G)$ which has discrete spectrum.
Then $\Yb \times \Yb$ is not ergodic and therefore also $\Xb \times \Yb$,
which naturally maps onto $\Yb \times \Yb$, is not ergodic.
Whence $\Xb$ is not WM.

Conversely if $\Xb$ is not WM  then there exists a nontrivial probability measure preserving
system $\Yb$ such that $\Xb \times \Yb$ is not ergodic.
Let $W \subset X \times Y$ be a measurable invariant subset of $\mu \times \nu$ measure
strictly between zero and one.
It is easy to check that the map $\phi : X \to L_2(Y, \nu);\ 
x \mapsto 1_{W_x}$, with
$W_x = \{y \in Y : (x,y) \in W\}$, is an equivariant, measurable, nontrivial map
from $X$ to $L_2(Y, \nu)$, where the action of $G$ on $L_2(Y, \nu)$ is via
the (unitary) Koopman representation $g \mapsto U_g, \ g \in G$. 
Let $\la = \phi_*(\nu)$ be the push-forward probability measure on $L_2(Y, \nu)$.

Clearly $\la$ is $m$-stationary and ergodic. 
By the ergodicity of $\la$ it follows that the invariant function $x \mapsto \|x\|$ is 
a constant $\la$ a.e. and we can therefore assume that $Z : = \supp(\la)$,
where the closure is taken with respect to the weak topology,
is contained in the weakly compact unite ball of $L_2(Y,\nu)$.
Next we note that the action $(Z, \la, \{U_g\}_{g \in G})$
is topologically a WAP system (see e.g. \cite[Sections 1.9 and 3.1]{G}) and then, applying \cite[Theorem 7.4]{FG1}, we conclude that the $\la$ is $G$-invariant. 
Now in a topologically transitive WAP system the invariant measure is unique and has discrete spectrum and our prof is complete.
\end{proof}

The assumption that $\Xb$ is an $m$-stationary system, and not merely a nonsingular one,
is really necessary as the following example shows.

\begin{exa}
{\em There are ergodic, conservative, nonsingular $\Z$-systems which are 
not WM yet do not admit a nontrivial measure preserving factor with discrete spectrum.}
Explicitly, consider a non-singular dyadic adding machine $(\Om,\mu_p,T)$ for $1/2 \not = p \in (0,1)$.
Here $\Om = \{0,1\}^\N$, $T\om = \om +\ch$ with $\ch = (1,0,0, \dots)$ and
$\mu_p = \{p, 1 - p\}^\N$ .
This nonsingular $\Z$-system is conservative, ergodic, and has no absolutely continuous invariant measure (see \cite{Aa} pages 29-31).
Clearly the product system$(\Om,\mu_p) \times (\Om, \mu_{1/2})$ is not ergodic, as
$\Om \times \Om$ decomposes into the disjoint union of $T \times T$-invariant graphs
$\Ga_\eta = \{(\om, \om + \eta) : \om \in \Om\}, \ \eta \in \Om$.
Since the system $(\Om, \mu_{1/2},T)$ is measure preserving it follows
that $(\Om, \mu_p,T)$ is not WM.
However, $(\Om, \mu_p,T)$ admits no nontrivial measure preserving factors. 

To see this assume to the contrary that
$\pi : (\Om, \mu_p,T) \to \Yb = (Y, \nu, S)$ is a measure preserving factor.
We consider two cases:

Case I : The system $\Yb$ has the property that 
$S^{2^n}$ acts ergodically for every $n \ge 1$.

In this case, for every $n \ge 1$ each of the $2^n$ ergodic components of $T^{2^n}$
is mapped onto $Y \pmod 0$. This however will contradict the fact that 
$\pi$ is
a homomorphism of nonsingular systems, unless $\Yb$ is trivial.

Case II : There is some $n \ge 1$ for which $(Y,\nu,S^{2^n})$ is not ergodic. Consider the smallest such $n$. 
Denoting $R = S^{2^{n-1}}$ we have that
$R$ acts ergodically and that $R^2 = S^{2^n}$. This implies that $S^{2^n}$
has exactly two ergodic components, say $A$ and $RA$. 
As $S$ commutes with $S^{2^n}$ we have that 
$SA$ is either $A$ or $RA$. Since $S$ is ergodic the first possibility is
ruled out and we have $SA = RA$, whence  $S^2 A = A$.
The conclusion is thus that already $S^2$ is not ergodic, with two
ergodic components $A$ and $SA$ each of $\nu$ measure $1/2$.
Now, the two $T^2$ ergodic components, say $B = [0]$ and $TB = [1]$,
whose $\mu$ measures are $p$ and $1-p$ are
mapped by $\pi$ onto $A$ and $SA$.
As $\pi_*(\mu_p) = \nu$, this implies that $p = 1 - p = 1/2$,
contradicting our assumption.
Thus this case is impossible and our proof is complete.
\end{exa}

The next result 
is essentially due to Kaimanovich \cite{Ka3}.
He proves part (2) for the Poisson boundary $\Pi(G,m)$, but then the result holds for
all its factors as well. Of course (2) implies (1).
(See also \cite{AL}.)

\begin{thm}\label{x^2}
Let $m$ be a non-degenerate spread out and symmetric probability measure on $G$.
Let $\mu$ be an $m$-stationary probability measure on $X$ such that
the $m$-stationary system $\Xb=(X,\Bcal,\mu,G)$ is $m$-proximal. 
\begin{enumerate}
\item
The nonsingular system $\Xb$ is doubly ergodic.
\item 
The product system $\Xb^2$ is EIC.
\end{enumerate} 
\end{thm}

Theorem \ref{x^2} is, in fact, a generalization of Kaimanovich' theorem.
Actually the notion of EIC is not even defined in Kaimanovich' paper \cite{Ka3}.
However, his proof 
can be easily modified to prove 
a stronger statement when the target is a separable metric space $(Z,d)$.

In order to understand the way this is done we first recall \cite[Theorem 6]{Ka3}.

\begin{thm}\label{Ka6}
Let $T : (\Om, \mu) \to (\Om, \mu)$ be the bilateral Bernoulli shift over a probability 
space $(X, \nu)$ (i.e. $\Om = X^\Z,\  \mu = \nu^\Z$ and for
$\om = ( \dots, x_{-1}, \dot{x}_0, x_1, \dots),\ (T(\om))_i = x_{i+1}$). 
If $E$ is a separable Banach space, and 
$f : \Om \to E$ and $\pi : X \to \Iso(E)$ are measurable maps such that a.e.
$f(T\om) = \pi (x_1) f(\om)$, then $f$ is a.e. a constant.
\end{thm}

Note the crucial assumption that $\pi$ depends only on the first coordinate $x_1$
of $\om$.
We claim that the same statement holds when we replace $E$ by a 
separable metric space $(Z,d)$.
In fact, all one needs to do is to assume (with no loss in generality)
that the metric $d$ is bounded, and then to replace the (linear) space 
$L_1(\Om,\mu,E)$ with norm
$$
\vertiii{f} =
 \int \|f(\om)\| \,d\mu(\om),
$$
which is used by Kaimanowich in his proof of Theorem \ref{Ka6},
by the (non-linear) space 
$M_1(\Om,\mu,Z)$ of all measurable functions $f : \Om \to Z$ with the metric
$$
D(f_1,f_2) = \int d(f_1(\om),f_2(\om))\, d\mu(\om).
$$
Now from this stronger version of  
\cite[Theorem 6]{Ka3}, theorem \ref{x^2} follows exactly as
\cite[Theorem 17]{Ka3} is derived from \cite[Theorem 6]{Ka3}.


\begin{exa}
Consider again the nonsingular SAT system $(\R,\mu,\Ga)$ described in Example \ref{SATnDE} above.
We claim that {\em there is no non-degenerate and symmetric probability measure $m$
on $\Ga$ which admits an $m$-stationary probability measure, say $\nu$, in the class of $\mu$ (i.e. equivalent to 
Lebesgue measure on $\R$)}.  In fact if $m$ on $G$ and $\nu$ on $\R$ are such measures then,
by \cite[Proposition 2.2]{J}, the system $(\R,\nu,\Ga)$ would be $m$-proximal, hence, by 
Theorem \ref{x^2}, this will ensure that the systems $(\R,\nu,\Ga)$ and therefore also
$(\R,\mu,\Ga)$ are DE, in contradiction to the claim in Example \ref{SATnDE}.
\end{exa}

\br

\section{The Poisson boundary for the free group $\mathbb{F}_2$}
\begin{exa}\label{exa,F2}
Let $G$ be the free group on two
generators, $G=\mathbb{F}_2=\langle a,b \rangle$, and $m$ the probability measure $m=\frac
14(\del_a+\del_b+\del_{a^{-1}}+\del_{b^{-1}})$.
Evidently $m$ is spread out, non-degenerate, and symmetric.
Let $Z$ be the space of right infinite reduced words on the letters
$\{a,a^{-1},b,b^{-1}\}$. $G$ acts on $Z$ by concatenation on the left
and reduction. Let $\eta$ be the probability measure on $Z$ given by
$$
\eta(C(\ep_1,\dots,\ep_n))=\frac1{4\cdot 3^{n-1}},
$$
where for $\ep_j\in\{a,a^{-1},b,b^{-1}\}$,\
$C(\ep_1,\dots,\ep_n)=\{z\in Z:z_j=\ep_j,\ j=1,\dots,n\}$. The
measure $\eta$ is $m$-stationary and the $m$-system ${\bf{Z}}
=(Z,\eta,G)$ is $m$-proximal. In fact ${\bf{Z}}$ is the Poisson
boundary $\Pi(\mathbb{F}_2,m)$. In particular then the system $\Zb$ is SAT
(see e.g. \cite{FG2}, Proposition 3.7), and by \cite{Ka1} it is doubly ergodic.
It now follows, by Theorem \ref{x^2}, that $\Xb^2$ is EIC.

It is not hard to see that if we let $\mathbb{F}_2$ act on the unit circle $\T =\{z \in \C : |z|=1\}$
via two appropriately chosen Moebius transformations $T_a$ and $T_b$, then the quotient
system $\Xb=(X,\mu,G)$, 
where $X = \mathbb{P} \cong \T/\{\pm 1\}$ is the projective line and 
$\mu$ denotes the image of Lebesgue's measure on $\T$, is isomorphic to $\Zb = (Z,\eta,G)$.

By choosing four disjoint intervals in $\mathbb{P}$ we can easily see that the system
$\Xb^2$ is not doubly ergodic (i.e. the diagonal $G$-action on $X \times X \times X \times X$
is not ergodic). Namely, if $A, B, C, D$ are disjoint arcs ordered counterclockwise
on the circle then, for any Moebius transformation $T$, the arcs 
$T A, T B, TC, TD$ are ordered either clockwise or counterclockwise. 
Thus e.g. 
$$
(T \times T \times T \times T) (A \times B \times C \times D) 
\cap (A \times C \times B \times D)  \not = \emptyset,
$$
can not be achieved.
Another way of proving this statement is to observe that the cross-ratio relation 
is preserved by the diagonal $G$-action on $X^4$.

As a conclusion we have:

\begin{prop}\label{F2exa}
The $G$-system $\Xb^2$ is EIC but not DE.
\end{prop}
\end{exa}

\br

\section{A $\Z$-system which is  EIC but not DE}\label{Sec,Z}

\begin{prop}\label{Zexa}
There exists a $\Z$-system which is EIC but not DE.
\end{prop}

\begin{proof}
Let $\{Y_n\}$ be a countable state Markov
chain on the state space $A$ which is conservative and
ergodic with an infinite invariant measure that has no
periodic factor. 
We claim that the corresponding shift (as a non-singular transformation) is EIC.
For the proof suppose that $f$ is a measurable map from
$Y = A^\Z$ to a separable metric space $(Z,d)$ such that
$$
f(Ty)=Uf(y)
$$
where $T$ is the shift and $U$ is an isometry of $Z$. 
Now fix a state $a_0$
in $A$ and look at the induced transformation on the cylinder set $[a_0]$.
This becomes a Bernoulli shift where the states are now all the
possible blocks $a_0,a_1,a_2,\dots,a_k$ that represent the states that
one visits before the next return to $a_0$. Let $B$ denote
this countable collection of blocks, and of course there
is a probability distribution on these blocks given by
the transition probabilities of the Markov chain.
If $S$ represents the induced transformation then $S$ is the Bernoulli shift
on $X = B^\Z$. Now if we call the restriction of $f$ to the cylinder set $[a_0]$,
$g$ then $g$ will satisfy:
$$
g(Sx)=U^{k(x_0)}g(x)
$$
where $k$ is the length of the block $x_0 \in B$. 
This is exactly the situation
of \cite[Theorem 6]{Ka3} adapted to isometries as in Section \ref{sec,m} 
(except for the inessential change that $x_1$ is replaced by $x_0$,
the $0$-th coordinate of $x$).
It follows that $g$ is a constant a.e. 
The essential range of $f$ is now a countable set on which $U$
will act as a permutation and, since the chain is assumed to be ergodic and conservative,
this set would have to be finite contradicting the fact that we assumed that the chain has no
periodic factor.

Now any such chain whose cartesian square is not conservative will give an example
of an EIC system which is not DE. 
For a simple example consider the simple symmetric planar random walk on $\Z^2$ with a positive probability of remaining at $(0,0)$. This is to eliminate the two point factor coming from the 
parity of the site. As is well known its cartesian square, which is the random walk on $\Z^4$,
is not recurrent, i.e. is not conservative. 
\end{proof}

\br

\section{$F_m$-proximality and mean proximality}
In this section we reconsider 
some notions of proximality introduced by Furstenberg in \cite{F}. 
Let $m$ be a  probability measure on $G$. For $n \ge 1$ set
$$
m_n =\frac1n (m + m^{(2)} + \cdots + m^{(n)}),
$$
where $m^{(j)} = m*m*\cdots *m$, ($j$-times).
Given a compact metric $G$-space $X$
we define the convolution operators $\A_n : C(X) \to C(X)$ by the formula
$$
\A_n f (x) = (m_n * f)(x) = \int f(gx) \, dm_n(g).
$$
We write $\A$ for $\A_1$.

The operator $\A$ is a {\em Markov operator} on $C(X)$; i.e. it is linear, positive
and satisfies $\A 1_X = 1_X$. Any Markov operator admits an invariant probability
measure, and it is called {\em uniquely ergodic} if there is only one invariant
probability measure (see \cite[Section 5.1]{Kr}).
The proof of the next theorem is almost verbatim the same as that 
of Theorem 4.9 in \cite{G}.

\begin{thm}\label{uniqe-stationarity}
The following conditions are equivalent for a Markov operator
$\A$ on a compact metric $G$-space $X$:
\begin{enumerate}
\item
There is a unique $m$-stationary probability measure on $X$;
i.e. $\A$ is uniquely ergodic.
\item
$C(X)=\R+\bar{B}$, where $B=\{f- \A f : f \in C(X)\}$ and $\bar B$ is
its closure in the topology of uniform convergence on $X$.
\item
For every continuous function $f\in C(X)$
the sequence of functions $\A_n f$
converges uniformly to a constant function $f^*$.
\item
For every continuous function $f\in C(X)$
the sequence of functions $\A_n f$ converges pointwise to a
constant function $f^*$.
\item
For every function $f \in \Acal$, for a collection
$\Acal \subset C(X)$ which linearly spans a uniformly dense
subspace of $C(X)$, the sequence of functions $\A_n f$
converges pointwise to a constant function.
\end{enumerate}
\end{thm}


\begin{defn} 
A compact metric $G$-space $X$ is called {\em $F_m$-proximal} if for each $x,y \in X$,
$m_n\{g : d(gx,gy) > \ep\} \to 0$ as $n \to \infty$ for any $\ep >0$.
\end{defn}

\begin{rmk}
In his paper \cite{F} Furstenberg calls this property $m$-proximality.
Unfortunately in some later works (\cite{FG1}, \cite{FG2}) 
the name ``$m$-proximal" was given the meaning we also adopt here, that of a boundary
(see the paragraph following Theorem 4.1 above).
Thus the reader is warned that $F_m$-proximal stands in this work for
the notion of $m$-proximal in \cite{F}. 
\end{rmk}

The next theorem is from \cite[Theorem 14.1]{F}.

\begin{thm}\label{Fur}
The following are equivalent for a compact metric $G$-space $X$:
\begin{enumerate}
\item
$X$ is $F_m$-proximal.
\item
Any solution $\nu \in M(X  \times X)$ to $m *\nu = \nu$ is concentrated on the diagonal 
$\Del(X) \subset X \times X$.
\item
For any $m$-stationary measure $\nu \in M(X)$ the $m$-stationary system $(X,\nu)$ is $m$-proximal.
\item
For any $\theta \in M(X)$ and $\ep >0$,
$$
\lim_{n \to\infty} m_n\{g \in G : d(g\theta,\del_X) > \ep \} = 0.
$$
\end{enumerate}
\end{thm}

To this list we now add the following:

\begin{thm}
The following conditions are equivalent for a compact metric $G$-space $X$:
\begin{enumerate}
\item
$X$ is $F_m$-proximal.
\item
There is a unique $m$-stationary probability measure on $X$, say $\nu$
(i.e. the $G$-space $X$ is uniquely $m$-ergodic),
and the $m$-stationary system $(X,\nu)$ is $m$-proximal.
\end{enumerate}
\end{thm}

\begin{proof}
(1) $\imp$ (2)\ 
Let $\nu_1$ and $\nu_2$ be two $m$-stationary measures on $X$.
We consider the ``natural" joining 
$$
\nu = \nu_1 \curlyvee \nu_2
=\int (\nu_1)_\om \times (\nu_2)_\om \, dP(\om),
$$
an element of $M(X \times X)$ (see \cite{FG1}, Section 3).
Clearly $m*\nu =\nu$ (with respect to the 
diagonal $G$-action) and by item (2) of the previous theorem we conclude that
$\nu$ is concentrated on the diagonal $\Del(X)$. This clearly implies that $\nu_1 = \nu_2$.

(2) $\imp$ (1)\ 
As we assume that $\nu$ is the unique $m$-stationary measure on $X$ and that $(X,\nu)$ is $m$-proximal, Theorem \ref{Fur} (3) implies that $X$ is $F_m$-proximal.
\end{proof}

\begin{defn}
A compact metric $G$-space $X$ is called {\em mean proximal} if it is $F_m$-proximal
for every $m \in M(G)$ whose closed support is $G$.
\end{defn}

\begin{exa}
Since $\mathbb{F}_2$ appears as a finite index subgroup of the group $PSL_2(\Z)$,
it follows from \cite[Theorem 16.8]{F} that the $G$-space $(\mathbb{P},\mathbb{F}_2)$ is mean proximal.
(See also \cite[Ch. VI.2]{Ma}.)
\end{exa}


\begin{thebibliography}{MMMMM}

\bibitem{Aa}
J. Aaronson,
{\em An introduction to infinite ergodic theory\/},
Math.\ Surveys and Monographs
{\bfseries 50}, Amer.\ Math.\ Soc.\ 1997.

\bibitem{AL}
J. Aaronson and M. Lema\'{n}czyk
{\it Exactness of Rokhlin endomorphisms and weak mixing of Poisson boundaries},
Algebraic and topological dynamics, 77-- 87, Contemp. Math., 
{\bf 385}, Amer. Math. Soc., Providence, RI, 2005. 

\bibitem{ALW}
J. Aaronson, M. Lin, and B. Weiss, 
{\em Mixing properties of Markov operators and ergodic transformations, and ergodicity of Cartesian products},
Israel J. Math. {\bf 33}, (1979), no. 3-4, 198--224, (1980).

%
\bibitem{BF}
U. Bader and A. Furman,
{\em Boundaries, Weyl groups, and superrigidity},
Arxiv:1109.3482v1, Sep. 2011.

\bibitem{Be}
A. Beck, 
{\em Eigen operators of ergodic transformations},
Trans. Amer. Math. Soc. {\bf 94}, (1960), 118--129. 

\bibitem{BK}
H. Becker and A.S.  Kechris, 
{\em The descriptive set theory of Polish group actions}, 
London Mathematical Society Lecture Note Series, {\bf 232},
Cambridge University Press, Cambridge, 1996.

\bibitem{Bj}
M. Bj\"{o}rklund,
{\em Five remarks about random walks on groups},
arXiv:1406.0763.

\bibitem{BFMS}
A. Bowels, L. Fidkowski, A. E. Marinello, and C. E. Silva,  
{\em Double Ergodicity of Infinite Transformations}, 
Illinois J. Math.,  {\bf 45}, (2001), 999--1019.
   
\bibitem{BM}
M. Burger and N. Monod, 
{\em Continuous bounded cohomology and applications to rigidity theory}, 
Geom. Funct. Anal., {\bf 12}, (2002), no. 2, 219--280.

\bibitem{D}
J. Dixmier,
{\em $C^*$-algebras}, North Holland, Amsterdam, 1982.

\bibitem{F1}
H. Furstenberg,
{\em Random walks and discrete subgroups of Lie groups\/},
Advances in probability and related topics, Vol {\bfseries 1},
Dekkers, 1971, pp. 1--63.

\bibitem{F}
H. Furstenberg,
{\em Boundary theory and stochastic processes on homogeneous spaces}. 
Harmonic analysis on homogeneous spaces (Proc. Sympos. Pure Math., Vol. XXVI, 
Williams Coll., Williamstown, Mass., 1972), 193--229. 
Amer. Math. Soc., Providence, R.I., 1973.
 
\bibitem{FG1}
H. Furstenberg and E. Glasner,
{\em Stationary dynamical systems\/},
Contemporary Math.
{\bfseries 532}, (2010), 1--28.

\bibitem{FG2}
H. Furstenberg and E. Glasner, 
{\em Recurrence for stationary group actions}. 
From Fourier analysis and number theory to radon transforms and geometry, 
283--291, Dev. Math., {\bf 28}, Springer, New York, 2013.

\bibitem{G}
E. Glasner,
{\em Ergodic theory via joinings\/},
AMS, Surveys and Monographs, {\bf{101}}, 2003.

\bibitem{Glu}
A. Gl\"{u}cksam,
{\em Ergodic multiplier properties},
arxiv.org:1306.3669

\bibitem{GIILS}
I. Grigoriev, N. Ince, M. C. Iordan, A. Lubin, and C. E. Silva, 
{\em On $\mu$-compatible metrics and measurable sensitivity}, 
Colloquium Math., {\bf 126}, (2012), 53--72.

\bibitem{JKLSS}
J. James, T. Koberda, k. Lindsey, P. Speh and C. E. Silva, 
{\em Measurable Sensitivity},  
Proc. Amer. Math. Soc.,  {\bf 136},  (2008), no. 10, 3549--3559.

\bibitem{J}
W. Jaworski,  
{\em Strongly approximately transitive group actions, the Choquet-Deny theorem, and polynomial growth}, 
Pacific J. Math., {\bf 165},  (1994), 115--129.

\bibitem{Ka1}
V. Kaimanovich,
{\em The Poisson boundary of covering Markov operators},
Israel J. of Math., {\bf 89}, (1995), 77--134.

\bibitem{Ka2}
V. A. Kaimanovich,
{\em SAT actions and ergodic properties of the horosphere foliation}, Rigidity in dynamics and geometry (Cambridge, 2000), 261--282, Springer, Berlin, 2002. 

\bibitem{Ka3}
V. A. Kaimanovich, 
{\em Double ergodicity of the Poisson boundary and applications to bounded cohomology}, 
Geom. Funct. Anal. {\bf 13}, (2003), no. 4, 852--861.

\bibitem{KV}
V. A. Kaimanovich and A. M. Vershik,
{\em Random walks on discrete groups:
boundary and entropy\/},
Ann.\ Probab.\  {\bfseries 11},
(1983), 457--490.

\bibitem{Ke}
A. S. Kechris,
{\em Classical descriptive set theory\/},
Springer-Verlag, Graduate texts in mathematics
{\bfseries 156}, 1991.

\bibitem{Kr}
U, Krengel, 
{\em Ergodic theorems}. 
With a supplement by Antoine Brunel. 
de Gruyter Studies in Mathematics, {\bf 6}. 
Walter de Gruyter \& Co., Berlin, 1985.

\bibitem{Mac}
G. H. Mackey,
{\em The theory of unitary group representations}, 
[an adapted version of the 1955 ``Chicago Notes"], Univ. 
of Chicago Press, Chicago, 1976.

\bibitem{Ma}
G. A. Margulis,
{\em Discrete subgroups of semisimple Lie groups}, 
Ergebnisse der Mathematik und ihrer Grenzgebiete (3) 
[Results in Mathematics and Related Areas (3)], {\bf 17}. Springer-Verlag, Berlin, 1991.

\bibitem{Mo}
N. Monod, 
{\em Continuous bounded cohomology of locally compact groups}, 
Lecture Notes in Mathematics, {\bf 1758}, Springer-Verlag, Berlin, 2001.

\bibitem{MS}
N. Monod and Y. Shalom, 
{\em Cocycle superrigidity and bounded cohomology for negatively curved spaces}, 
J. Differential Geom. {\bf 67}, (2004), no. 3, 395--455.

\bibitem{M}
C. C. Moore,
{\em Groups with Finite Dimensional Irreducible Representations},
Trans. Amer. Math. Soc.,
{\bf 166}, (1972),  401--410.

\bibitem{NZ}
A. Nevo and R. J. Zimmer,
{\em Homogenous projective factors for
actions of semi-simple Lie groups\/},
Invent.\ Math.\ {\bfseries 138}, (1999), 229-252.

\bibitem{T}
J. Tomiyama, 
{\em Invitation to $C^*$-algebras and topological dynamics}, 
World Scientific, 1987.

\bibitem{V}
V. S. Varadarajan,  
{\em Groups of automorphisms of Borel spaces}, 
Trans. Amer. Math. Soc. {\bf 109}, (1963),  191--220. 

\end{thebibliography}
\end{document}